\documentclass[10pt,a4paper]{amsart}

\usepackage{amsmath,amssymb,amsthm}

\newtheorem{Thm}{Theorem}[section]

\newtheorem{Lem}[Thm]{Lemma}
\newtheorem{Prop}[Thm]{Proposition}
\newtheorem{Cor}[Thm]{Corollary}
\newtheorem{Conj}[Thm]{Conjecture}

\theoremstyle{remark}
\newtheorem{Rem}[Thm]{Remark}

\begin{document}

\title[On the modularity of elliptic curves]{modularity of elliptic curves over abelian totally real fields unramified at 3, 5, and 7}
\author{Sho Yoshikawa}
\address{Graduate School of Mathematical Sciences,
the University of Tokyo, 3-8-1, Komaba, Meguro-Ku, Tokyo 153-8914, Japan}
\email{yoshi@ms.u-tokyo.ac.jp}

\maketitle

\begin{abstract}
Let $E$ be an elliptic curve over an abelian totally real field $K$ unramified at 3,5, and 7.
We prove that $E$ is modular.
\end{abstract}

\section{Introduction}

Let $E$ be an elliptic curve over a totally real field $K$. We say that $E$ is modular if there exists a Hilbert eigenform $f$ over $K$ of parallel weight $2$ such that $L(E,s)=L(f,s)$. The classical Shimura-Taniyama conjecture asserts that all elliptic curves over $\mathbb{Q}$ are modular. This conjecture for semistable elliptic curves, which was the crucial step in proving Fermat's Last Theorem, was proved by Wiles \cite{Wi} and Tayor-Wiles \cite{TW}. Later, the general case of the conjecture was completed by Breuil-Conrad-Diamond-Taylor \cite{BCDT}.

The Shimura-Taniyama conjecture has a natural generalization to totally real fields:

\begin{Conj}
\label{ST}
Let $K$ be a totally real number field. Then, any elliptic curve over $K$ is modular.
\end{Conj}

 Recently, a breakthrough on this problem was brought by Freitas-Le Hung-Siksek. In their paper \cite{FLS}, they prove Conjecture \ref{ST} for any quadratic field.
The aim of this paper is to attack Conjecture \ref{ST} for abelian totally real fields. More precisely, the main theorem is the following:

\begin{Thm}
\label{mainab}
Let $K$ be a totally real number field which is abelian over $\mathbb{Q}$.
Suppose that $K$ is unramified at every prime above 3,5, and 7.
Then, any elliptic curve over $K$ is modular.
\end{Thm}

\begin{Rem}
In contrast to the proof of \cite[Theorem 1]{FLS}, we do not use geometric techniques on modular curves.
Instead, our proof involves the modularity lifting theorem \cite{SW1} in an essential way, and hence at present we cannot remove the hypothesis that $K$ is abelian.
\end{Rem}

In the rest of this section, we describe the structure of the proof of Theorem \ref{mainab}.\\

Firstly, we prove the following proposition, which is a complementary result of \cite[Theorem 7]{FLS}; we consider elliptic curves with additive reduction at a prime dividing $p=5$ or $7$ instead of semi-stable reduction. The proof is given in Section \ref{pfmainadd}.

\begin{Prop}
\label{mainadd}
Let $p=5$ or $7$.
Let $K$ be a totally real field, $\mathfrak{p}$ a prime of $K$ dividing $p$, and $E$ an elliptic curve over $K$.
Assume that $K$ is unramified at $\mathfrak{p}$ and that $E$ has additive reduction at $\mathfrak{p}$ with $\bar{\rho}_{E,p}: \mathrm{Gal}(\bar{K}/K)\rightarrow \mathrm{GL}_2(\mathbb{F}_p)$ (absolutely) irreducible.
Then $E$ is modular, unless either of the following exceptional cases hold:
\begin{enumerate}
\item $p=5$, $v_\mathfrak{p}(j_E)\equiv 1$ mod $3$, and $E$ has additive potential good (supersingular) reduction at $\mathfrak{p}$, or
\item $p=7$, $v_\mathfrak{p}(j_E)\equiv 2$ mod $3$, and $E$ has additive potential good (ordinary) reduction at $\mathfrak{p}$.
\end{enumerate} 
\end{Prop}

Here, $j_E$ denotes the $j$-invariant of $E$, and $v_\mathfrak{p}$ is the normalized discrete valuation of $K$ at $\mathfrak{p}$.
Also,  $\bar{\rho}_{E,p}$ denotes the mod $p$ Galois representation defined by the $p$-torsion points of $E$.  Note that, for $p\neq 2$, $\bar{\rho}_{E,p}$ is irreducible if and only if $\bar{\rho}_{E,p}$ is absolutely irreducible: This follows from the presence of the complex conjugates in $G_K$.

The basic strategy for the proof of Proposition \ref{mainadd} is based on \cite[Theorem 7]{FLS}, but because we treat the cases of additive reduction, we need to look at local mod $p$ Galois representations more carefully. The local computations are carried out in Section \ref{local}. For this, we heavily use the results of Kraus \cite{Kra}.\\

Secondly, combining Proposition \ref{mainadd} and \cite[Theorem 7]{FLS}, we prove the following theorem.
The proof is given in Section \ref{pfmain2}.

\begin{Thm}\label{main2}
Let $K$ be a totally real field in which 7 is unramified.
If $E$ is an elliptic curve over $K$ with $\bar{\rho}_{E,7}$ (absolutely) irreducible, then $E$ is modular.
\end{Thm}

Theorem \ref{main2} is seen as a mod 7 variant of the following theorem due to Thorne:
\begin{Thm}\label{Thorne}
(\cite[Theorem 7.6]{Th})
Let $K$ be a totally real field with $\sqrt{5}\notin K$.
If $E$ is an elliptic curve over $K$ with $\bar{\rho}_{E,5}$ (absolutely) irreducible, then $E$ is modular.
\end{Thm}

Finally, we apply Theorem \ref{main2} and Theorem \ref{Thorne} to prove Theorem \ref{mainab}; we show that, for an elliptic curve $E$ which is not yet known to be modular, a quadratic twist of $E$ becomes semi-stable at all primes dividing 3, in which case we already know its modularity by \cite{Fr}.
The detail is given in Section \ref{pfmainab}.

\section{Local computations}
\label{local}

First, we fix the notation of this section:
\begin{enumerate}
\item $p$ is a prime number.
\item $F$ is an absolutely unramified $p$-adic local field.
\item $E$ is an elliptic curve over $F$. 
\item $v$ is the normalized $p$-adic discrete valuation of $F$. 
\item $\omega_1:I\rightarrow \mu_{p-1}(\bar{F})\rightarrow \mathbb{F}_p^\times$ denotes the fundamental character of level $1$, and $\omega_2, \omega_2':I\rightarrow \mu_{p^2-1}(\bar{F})\rightarrow \mathbb{F}_{p^2}^\times$ denote the fundamental characters of level $2$. Here, $I$ is the inertia subgroup of $G_F$.
\end{enumerate}

For the proof of Proposition \ref{mainadd}, we consider only the cases when $E$ has additive reduction; that is, $E$ has potential multiplicative reduction, potential good ordinary reduction, or potential good supersingular reduction.
In the following  subsection, we treat these three cases separately, and we review the results of Kraus in \cite{Kra} without proof.
We note that, although Kraus proves his results for elliptic curves over $\mathbb{Q}_p$, the proofs also work without change for those over any absolutely unramified $p$-adic field.

\subsection*{Potential multiplicative reduction case}\label{pm}

\begin{Prop}
\label{semist}
Let $p\geq 3$ be a prime number, $F$ an unramified extension of $\mathbb{Q}_p$, and $E$ an elliptic curve over $F$ with additive potential multiplicative reduction. 
Then, the restriction of $\bar{\rho}_{E,p}$ to the inertia subgroup $I$ is of the form
\begin{equation}
\label{Msemist}
 \bar{\rho}_{E,p}|_{I}\simeq 
\begin{pmatrix}
\omega_1^{\frac{p+1}{2}} & * \\
0 & \omega_1^{\frac{p-1}{2}} \\
\end{pmatrix}.
\end{equation}
\end{Prop}
\begin{proof}
See \cite[PROPOSITION 10]{Kra}. 
\end{proof}
Since the projective image of \eqref{Msemist} is of the form
$\begin{pmatrix}
\omega_1 & * \\
0 &1 \\
\end{pmatrix}$,
we obtain the following corollary:
\begin{Cor}
\label{Csemist}
In the setting of Proposition \ref{semist}, 
the projective image $\mathbb{P}\bar{\rho}_{E,p}(G_F)$ contains a cyclic subgroup of order $p-1$.
\end{Cor}

\subsection*{Potential ordinary reduction case}

\begin{Prop}
\label{ord}
Let $p\geq 5$ be a prime number, $F$ an unramified extension of $\mathbb{Q}_p$, and $E$ an elliptic curve over $F$ with additive potential ordinary reduction. 
Denote $\Delta$ for a minimal discriminant of $E$ and $v$ for the normalized discrete valuation of $F$.
Set $\alpha=(p-1)v(\Delta)/12$, which is an integer as noted just before 2.3.2 in \cite{Kra}.
Then, the restriction of $\bar{\rho}_{E,p}$ to the inertia subgroup $I$ is of the form
\begin{equation}
\label{Mord}
 \bar{\rho}_{E,p}|_I\simeq 
\begin{pmatrix}
\omega_1^{1-\alpha} & * \\
0 & \omega_1^{\alpha} \\
\end{pmatrix}.
\end{equation}
\end{Prop}
\begin{proof}
See \cite[PROPOSITION 1]{Kra}.
\end{proof}

The projective image of \eqref{Mord} is of the form
$\begin{pmatrix}
\omega_1^{1-2\alpha} & * \\
0 &1 \\
\end{pmatrix}$,
and $\omega_1^{1-2\alpha}$ is a character of order $m:=\frac{p-1}{(p-1,1-2\alpha)}$.
Thus, the projective image $\mathbb{P}\bar{\rho}_{E,p}(G_F)$ contains a cyclic subgroup of order $m$.
In the following, we compute the order $m$ for certain $p$, which we will take as $5$ or $7$ in Section \ref{pfmainadd}

Suppose first that $p$ is a prime number of the form $p=2^a+1$ for an integer $a\geq 2$.
Since $1-2\alpha$ is an odd integer, $1-2\alpha$ is prime to $p-1=2^a$ so that we have $m=p-1$.
Thus, we have the following corollary.

\begin{Cor}
\label{Cord5}
Let $p$ be a prime number of the form $p=2^a +1$ with $a\geq 2$ an integer, $F/\mathbb{Q}_p$ an unramified extension, and $E$ an elliptic curve over $F$ with additive potential good ordinary reduction.
Then, the projective image $\mathbb{P}\bar{\rho}_{E,p}(G_F)$ contains a cyclic group of order $p-1$.
\end{Cor}

Suppose next that $p$ is a prime number of the form $p=3\cdot 2^a+1$ with $a\geq 1$ an integer.
Since $\alpha=v(\Delta)/2$ is an integer, $1-2\alpha=1-v(\Delta)$ is odd.
Thus, we have
\[
  m = \begin{cases}
    \frac{p-1}{3} & (v(\Delta)\equiv 1\ \mathrm{mod}\ 3) \\
    p-1 & (\mathrm{otherwise}).
  \end{cases}
\]
Therefore, we obtain the following corollary:

%


%

\begin{Cor}
\label{Cord7}
Let $p$ be a prime number of the form $p=3\cdot 2^a+1$ for an integer $a\geq 1$, $F/\mathbb{Q}_p$ an unramified extension, and $E$ be an elliptic curve over $F$ with additive potential good ordinary reduction. 
Let also $\Delta$ be a minimal discriminant of $E$. 
Then, $\mathbb{P}\bar{\rho}_{E,p}(G_F)$ contains a cyclic group of order $(p-1)/3$ or $p-1$, depending on whether $v(\Delta)\equiv 1$ mod $3$ or $v(\Delta)\equiv 0,2$ mod $3$, respectively.
\end{Cor}

\subsection*{Potential supersingular reduction case}

As in the previous subsections, we begin with Kraus' result.

\begin{Prop}
\label{ss}
Let $p\geq 5$ be a prime number, $F$ an unramified extension of $\mathbb{Q}_p$, and $E$ an elliptic curve over $F$ with additive potential supersingular reduction. We choose a minimal model 
\[y^2=x^3+Ax+B\] 
of $E$.
Also, let $\Delta$ denote a minimal discriminant of $E$.
\begin{itemize}
\item[(a)]
If $(v(\Delta), v(A), v(B))$ is one of the triples $(2,1,1), (3,1,2), (4,2,2),$ $(8,3,4)$, $(9,3,5)$, or $(10,4,5)$, then $\bar{\rho}_{E,p}$ is wildly ramified.
\item[(b)]
If $(v(\Delta), v(A), v(B))$ is not any of the above triples, then the restriction of  $\bar{\rho}_{E,p}$ to the inertia subgroup $I$ is given by
\begin{equation}
\label{Mss}
 \bar{\rho}_{E,p}|_{I}\otimes \mathbb{F}_{p^2}\simeq  
\begin{pmatrix}
\omega_2^\alpha {\omega_2'}^{p-\alpha}  & 0 \\
0 & {\omega_2'}^\alpha {\omega_2}^{p-\alpha}  \\
\end{pmatrix}.
\end{equation}
Here, $\alpha=(p+1)v(\Delta)/12$ is an integer as noted in \cite[PROPOSITION 2]{Kra}.
\end{itemize}
\end{Prop}

\begin{proof}
The part (a) is a consequence of LEMME 2 and PROPOSITION 4 in \cite{Kra}.
The part (b) follows directly from PROPOSITION 2 and LEMME 2 in \cite{Kra}.
\end{proof}
%


From the case (a) in the above proposition, we immediately obtain the following corollary:

\begin{Cor}
\label{ssa}
Let the notation be as in Proposition \ref{ss}. If the condition of (a) holds, then the projective image $\mathbb{P}\bar{\rho}_{E,p}(G_F)$ contains a $p$-group.
\end{Cor}

Next, we consider the case (b) in the Proposition \ref{ss}.
The image of \eqref{Mss} in $\mathrm{PGL}_2(\mathbb{F}_{p^2})$ is of the form
$\begin{pmatrix}
\omega_2^{-(p-1)(2\alpha+1)}  & 0 \\
0 & 1  \\
\end{pmatrix}$.
Since the character $\omega_2^{-(p-1)(2\alpha+1)}$ is of order $n:=\frac{p+1}{(p+1,2\alpha+1)}$,
the projective image $\mathbb{P}(\bar{\rho}_{E,p}\otimes \mathbb{F}_{p^2})(G_F)$ (and hence $\mathbb{P}(\bar{\rho}_{E,p})(G_F)$) contains a cyclic subgroup of order $n$.
In the rest of this subsection, we make computations of the number $n$ for certain $p$.
We will apply them to the case $p=5$ or $7$ in Section \ref{pfmainadd}.

Suppose first that $p$ is a prime number of the form $p=2^a-1$ with $a\geq 3$ an integer.
Since $\alpha$ is an integer, $2\alpha+1$ is prime to $p+1=2^a$ so that $n=p+1$.
Thus, we have proved the following corollary:

\begin{Cor}
\label{ss7}
Let $p$ be a prime number of the form $p=2^a-1$ with $a\geq 3$ an integer, $F/\mathbb{Q}_p$ an unramified extension, and $E$ an elliptic curve over $F$ with additive potential good supersingular reduction.
Assume the condition of (b) in Proposition \ref{ss} holds.
Then, the projective image $\mathbb{P}\bar{\rho}_{E,p}(G_F)$ contains a cyclic group of order $p+1$.
\end{Cor}

 

Suppose next that $p$ is a prime number of the form $p=3\cdot 2^a-1$ with $a\geq 1$ an integer.
Since $\alpha = v(\Delta)/2$ is an integer, $2\alpha +1=v(\Delta)+1$ is odd.
Thus, we have
\[
  n = \begin{cases}
    \frac{p+1}{3} & (v(\Delta)\equiv 2\ \mathrm{mod}\ 3) \\
    p+1 & (\mathrm{otherwise}).
  \end{cases}
\]
%
Therefore, we obtain the following corollary:
\begin{Cor}
\label{ss5}
Let $p$ be a prime number of the form $p=3\cdot 2^a-1$ with $a\geq 1$ an integer, $F/\mathbb{Q}_p$ an unramified extension, and $E$ an elliptic curve over $F$ with additive potential good supersingular reduction. 
Let also $\Delta$ be a minimal discriminant of $E$. 
Assume the condition of (b) in Proposition \ref{ss} holds.
Then, $\mathbb{P}\bar{\rho}_{E,p}(G_F)$ contains a cyclic group of order $(p+1)/3$ or $p+1$, depending on whether $v(\Delta)\equiv 2$ mod $3$ or $v(\Delta)=0,1$ mod $3$, respectively.
\end{Cor}

\section{Proof of Proposition \ref{mainadd}} 
\label{pfmainadd}
Before proceeding to the proof of Proposition \ref{mainadd}, we cite two theorems from \cite{FLS}.

First, the following theorem is a consequence of a combination of modularity switching arguments and a modularity lifting theorem.
As remarked in \cite{FLS}, the modularity lifting theorem used there is relatively a direct consequence of the theorem of Breuil-Diamond, which is also an accumulation of works by many people.

\begin{Thm}
\label{mlt}
(\cite[Theorem 3 and Theorem 4]{FLS})
Let $E$ be an elliptic curve over a totally real field $K$.
If $p= 5$ or $7$, and if $\bar{\rho}_{E,p}|_{G_{K(\zeta_p)}}$ is absolutely irreducible, then $E$ is modular.
\end{Thm}


Second, the following result will be useful in our setting for proving that $\bar{\rho}_{E,p}|_{G_{K(\zeta_p)}}$ is absolutely irreducible.
\begin{Thm}
\label{projim}
(\cite[Proposition 9.1]{FLS})
Let $p= 5$ or $7$, and $K$ be a totally real field satisfying $K\cap \mathbb{Q}(\zeta_p)=\mathbb{Q}$.
For an elliptic curve $E$ over $K$ such that $\bar{\rho}_{E,p}$ is (absolutely) irreducible but $\bar{\rho}_{E,p}|_{G_{K(\zeta_p)}}$ is absolutely reducible, we have the following:
\begin{enumerate}
\item If $p=5$, then $\bar{\rho}_{E,5}(G_K)$ is a group of order $16$, and its projective image $\mathbb{P}\bar{\rho}_{E,5}(G_K)$ is isomorphic to $(\mathbb{Z}/2\mathbb{Z})^2$. 
\item If $p=7$, then $\mathbb{P}\bar{\rho}_{E,7}(G_K)$ is isomorphic to $S_3$ or $D_4$.
\end{enumerate}
\end{Thm}

We are now ready to prove Proposition \ref{mainadd}.
Let $p=5$ or $7$, $K$ a totally real field, and $\mathfrak{p}$ a prime ideal of $K$ dividing $p$.
Let $E$ be an elliptic curve over $K$.
Also, denote by $\Delta$ a minimal discriminant of $E_\mathfrak{p}:=E\otimes_K K_\mathfrak{p}$.

We make the following assumptions as in the statement of Proposition \ref{mainadd}:
\begin{enumerate}
\item $K$ is unramified at $\mathfrak{p}$, 
\item $\bar{\rho}_{E,p}$ (absolutely) irreducible, and
\item $E$ has additive reduction at $\mathfrak{p}$.
\end{enumerate}

We split the proof into three cases according to reduction of $E$:\\

(i) 
If $E$ has additive potential multiplicative reduction at $\mathfrak{p}$, then Corollary \ref{Csemist} for $E_\mathfrak{p}$ implies that $\mathbb{P}\bar{\rho}_{E,p}(G_K)$ has a cyclic subgroup of order $p-1$.
Thus, Theorem $\ref{projim}$ implies that  $\bar{\rho}_{E,p}|_{G_{K(\zeta_p)}}$ cannot be absolutely reducible. Hence $E$ is modular by Theorem \ref{mlt}.\\

(ii)
Suppose next that $E$ has additive potential ordinary reduction at $\mathfrak{p}$.\\
If $p=5$, then Corollary \ref{Cord5} for $E_\mathfrak{p}$ shows that $\mathbb{P}\bar{\rho}_{E,5}(G_K)$ contains a cyclic subgroup of order $4$. Thus, by Theorem \ref{projim}, $\bar{\rho}_{E,5}|_{G_{K(\zeta_5)}}$ is absolutely irreducible, whence $E$ is modular.

Also, if $p=7$ and $v(\Delta)\equiv 0,2$ mod $3$, then Corollary \ref{Cord7} shows that $\mathbb{P}\bar{\rho}_{E,7}(G_K)$ has a cyclic subgroup of order $6$. Hence, Theorem \ref{projim} implies that $\bar{\rho}_{E,7}|_{G_{K(\zeta_7)}}$ is absolutely irreducible, in which case $E$ is modular.

We consider the remaining cases; that is, $p=7$ and $v_\mathfrak{p}(\Delta)\equiv 1$. If $j_E\neq 0$, then these cases are equivalent to the case $v_\mathfrak{p}(j_E)\equiv 2$ modulo $3$; in fact, this follows by taking a minimal model $y^2=x^3+Ax+B$ of $E_\mathfrak{p}$ and noting that $j_E=1728A^3/\Delta$. If $j_E =0$,  then $E$ is modular since $E$ becomes defined over $\mathbb{Q}$ after suitable (solvable) base change.\\

(iii)
Finally, suppose that $E$ has additive potential supersingular reduction at $\mathfrak{p}$. 

If the condition (a) in Proposition \ref{ss} holds, then Corollary \ref{ssa},  Theorem \ref{projim}, and Theorem \ref{mlt} show that $E$ is modular.

Assume the condition (b) in Proposition \ref{ss} holds. Then we have the following two cases:
\begin{itemize}
\item If $p=5$ and $v(\Delta)\equiv 0,1$ mod $3$, then $\mathbb{P}\bar{\rho}_{E,5}(G_K)$ contains a cyclic subgroup of order 6 by Corollary \ref{ss5}.  Hence, Theorem \ref{projim} and Theorem \ref{mlt} show that $E$ is modular.
If $j_E\neq 0$, then the remaining case when $p=5$ and $v_\mathfrak{p}(\Delta)\equiv2$ mod $3$ can be rephrased as  $v_\mathfrak{p}(j_E)\equiv 1$ modulo $3$. If $j_E=0$, then $E$ is modular as in (ii).
\item If $p=7$, then $E$ is modular by  Corollary \ref{ss7} with Theorem \ref{projim}, and Theorem \ref{mlt}.
\end{itemize}

In summary, combining (i), (ii), and (iii), we see that $E$ is modular unless the following conditions hold:
\begin{enumerate}
\item $p=5$, $v_\mathfrak{p}(j_E)\equiv 1$ mod $3$, and $E$ has additive potential good (supersingular) reduction at $\mathfrak{p}$, or
\item $p=7$, $v_\mathfrak{p}(j_E)\equiv 2$ mod $3$, and $E$ has additive potential good (ordinary) reduction at $\mathfrak{p}$.
\end{enumerate}
This proves Proposition \ref{mainadd}.

\section{Proof of Theorem \ref{main2}}\label{pfmain2}
Let $K$ and $E$ be as in Theorem \ref{main2}.
If $E$ has semi-stable reduction at some prime dividing $7$, then the assertion follows from \cite[Theorem 7]{FLS}.
So suppose that $E$ has additive reduction at every prime $\mathfrak{p}|7$.
By Proposition \ref{mainadd} and Theorem \ref{mlt}, we have only to consider the case when $E$ has additive potential ordinary reduction at every prime $\mathfrak{p}|7$ and $\bar{\rho}_{E,7}|_{G(K(\zeta_7))}$ is absolutely reducible.
In this case, we may apply the modularity lifting theorem \cite{SW2} due to Skinner-Wiles in order to prove the modularity of $E$.

\begin{Rem}
A similar argument does not reprove Theorem \ref{Thorne} even if $K$ is just unramified at $5$;
in fact, Proposition \ref{mainadd} implies that an elliptic curve $E$ over $K$ with $\bar{\rho}_{E,5}|_{G(K(\zeta_5))}$ absolutely reducible must have additive potential supersingular reduction at every prime $\mathfrak{p}|5$. In such a case, the theorem of Skinner-Wiles \cite{SW2} is not available.
\end{Rem}

\begin{Rem}
In his thesis \cite{Bao}, Le Hung essentially shows the following;
if $K$ is a totally real field unramified at 5 and 7, and if $E$ is an elliptic curve over $K$ with both $\bar{\rho}_{E,p}$ ($p=5,7$) irreducible, then $E$ is modular.
This follows from \cite[Proposition 6.1]{Bao} combined with the modularity lifting theorem due to Skinner-Wiles \cite{SW2}.
\end{Rem}

\begin{Rem}
Very recently, S. Kalyanswamy \cite{Ka} announced to prove a version of Theorem \ref{main2}.
He actually proves a new modularity  theorem \cite[Theorem 3.4]{Ka} for certain Galois representations, and applies it to elliptic curves in \cite[Theorem 4.4]{Ka}.
For clarity, we describe the difference between Theorem \ref{main2} and \cite[Theorem 4.4]{Ka}: Kalyanswamy considers elliptic curves over a totally real field $F$ with $F\cap \mathbb{Q}(\zeta_7)=\mathbb{Q}$, which is weaker than the assumption that $F$ is unramified at 7, while he also imposes an additional condition on the mod 7 Galois representations. Therefore, both Theorem \ref{main2} and \cite[Theorem 4.4]{Ka} have their own advantage.
\end{Rem}

\section{Proof of Theorem \ref{mainab}}
\label{pfmainab}

For the proof of Theorem \ref{mainab}, we need another modularity theorem due to Freitas \cite{Fr}.
This theorem essentially follows from \cite{SW1}, \cite{SW2}, and Theorem \ref{mlt}.

\begin{Thm}
\cite[Theorem 6.3]{Fr}
\label{semi}
Let $K$ be an abelian totally real field where 3 is unramified. Let $E$ be an elliptic curve over $K$ semistable at all primes $\mathfrak{p}|3$.
Then, $E$ is modular.
\end{Thm}

Also, we note a well-known result of a torsion version of Neron-Ogg-Shafarevich. 

\begin{Lem}
\cite[VII, Exercises 7.9]{Sil}
\label{inertia}
Let $F$ be a local field, $E$ an elliptic curve over $F$ with potential good reduction, and $m\geq 3$ an integer relatively prime to the residual characteristic of $F$.
\begin{itemize}
\item[(a)] The inertia group of $F(E[m])/F$ is independent of $m$.
\item[(b)] The extension $F(E[m])/F$ is unramified if and only if $E$ has good reduction.
\end{itemize}
\end{Lem}

Now we are ready to prove Theorem \ref{mainab}.
\vspace{0.3cm}\\
\textit{Proof of Theorem \ref{mainab}.\ }
Let $K$ be as in Theorem \ref{mainab} and $E$ be an elliptic curve over $K$.
We prove that $E$ is modular.
By Theorem \ref{main2} and Theorem \ref{Thorne}, we may assume that both $\bar{\rho}_{E,5}$ and $\bar{\rho}_{E,7}$ are reducible;
that is, $\bar{\rho}_{E,p}$ for $p=5,7$ factors through a Borel subgroup $B(\mathbb{F}_p)$.
Note that $B(\mathbb{F}_5)$ (resp. $B(\mathbb{F}_7)$) is of order $4^2\cdot5$ (resp. $6^2\cdot7$).  

Let $\mathfrak{p}$ be a prime of $K$ dividing $3$.
By Lemma \ref{inertia} (a), the representations of the inertia subgroup $I_\mathfrak{p}$ on $E[5]$ and $E[7]$ factor through the same quotient $I_\mathfrak{p}'$.
Since $|I_\mathfrak{p}'|$ divides $\mathrm{gcd}(4^2\cdot5, 6^2\cdot7)=4$, $I_\mathfrak{p}'$ is tame (and so cyclic) of order dividing 4. 
Subgroups of order 4 in $B(\mathbb{F}_7)$  are not cyclic, and hence $I_\mathfrak{p}'$ is trivial or of order 2.
If $E_\mathfrak{p}=E\otimes K_\mathfrak{p}$ has potential good reduction, $E_\mathfrak{p}$ becomes of good reduction over at most a quadratic extension of $K_\mathfrak{p}$ by Lemma \ref{inertia} (b).
Also, in the case where $E_\mathfrak{p}$ has potential multiplicative reduction, $E_\mathfrak{p}$ becomes isomorphic to a Tate curve after a quadratic extension of $K_v$.
Summarizing both cases, we see that $E_\mathfrak{p}$ becomes semi-stable after a quadratic base change.
This implies that, if $E_\mathfrak{p}$ has bad reduction, a quadratic twist of $E_\mathfrak{p}$ (by a uniformizer of $K_\mathfrak{p}$) is semi-stable.
Using the weak approximation theorem, we find $d\in K$ such that the quadratic twist $E^{(d)}$ of $E$ by $d$ is semi-stable at every prime $\mathfrak{p}|3$.
Theorem \ref{semi} shows that $E^{(d)}$ is modular.
Since the modularity of elliptic curves is invariant under quadratic twists, it follows that $E$ is modular.
 \hfill $\Box$

\section*{Acknowledgement}
The author expresses his deepest gratitude to his advisor Prof. Takeshi Saito, who gave him many helpful comments to enhance the arguments in this paper and encouraged him during his writing this paper.
His thanks also goes to Bao Le Hung, who kindly answered to the author's question on Skinner-Wiles' theorem.

\end{document}